%% file: arrow-ce.tex
\documentclass[12pt]{amsart}
\pdfoutput=1
\usepackage[utf8]{inputenc}
\usepackage{MnSymbol}
\usepackage{accents}
\usepackage{tikz-cd}\usetikzlibrary{decorations.pathmorphing}

\usepackage{bm}
\usepackage[textsize=footnotesize,shadow]{todonotes}
\usepackage{fancyhdr}
\usepackage[colorlinks=true,unicode=true]{hyperref}
\usepackage{mathtools}
\usepackage{lpic}

\input{style.tex}
\input{defs.tex}

\def\draft{\fbox{\tiny draft: \today}}
\let\draft=\relax
\setcolors[rmcolor=cyan,
           commentcolor=purple,
           inscolor=blue,
           commentsymbol=$\boxed{\mathbb{S}}$,
           rmsymbol=$\boxed{\mathbb{S}_{rm}}$,
           inssymbol=$\boxed{\mathbb{S}_{ins}}$]
          {s}
\setcolors[rmcolor=green,
           commentcolor=red,
           inscolor=brown,
           commentsymbol=$\boxed{\mathbb{J}}$,
           rmsymbol=$\boxed{\mathbb{J}_{rm}}$,
           inssymbol=$\boxed{\mathbb{J}_{ins}}$]
          {j}

\title{Arrow Contraction and Expansion in Tropical Diagrams}
\author[RM]{R. Matveev}
\author[JWP]{J. W. Portegies}
\begin{document}
\thispagestyle{fancy} 
\begin{abstract}
Arrow contraction applied to a tropical diagram of probability spaces
is a modification of the diagram, replacing one of the morphisms by an
isomorphims, while preserving other parts of the diagram. It is
related to the rate regions introduced by Ahlswede and K\"orner. In a
companion article we use arrow contraction to derive information about
the shape of the entropic cone. Arrow expansion is the inverse
operation to the arrow contraction.
\end{abstract}

\maketitle

\son
\jon

\section{Introduction}
\input{section-intro}

\section{Preliminaries}
\label{se:preliminaries}
\input{section-preliminaries}

\section{Arrow Contraction}
\input{section-contraction}

\section{Local Estimate}
\input{section-local}

\section{Proof of Proposition~\ref{p:contraction-approximation}}\label{s:proof}
\input{section-proof}

\bibliographystyle{alpha}       
\bibliography{ReferencesACE}
\end{document}

%% file: section-intro.tex
In~\cite{Matveev-Asymptotic-2018} we have initiated the theory of
\term{tropical probability spaces} and in~\cite{Matveev-Tropical-Entropic-2019}
applied the techniques to derive a dimension-reduction result for the
entropic cone of four random variables. 

Two of the main tools used for the latter are what we call \term{arrow
contraction} and \term{arrow expansion}. They are formulated for
tropical commutative diagrams of probability spaces.  Tropical
diagrams are points in the asymptotic cone of the metric space of
commutative diagrams of probability spaces endowed with the asymptotic
entropy distance. Arrows in diagrams of probability spaces are
(equivalence classes of) measure-preserving maps.

Arrow contraction and expansion take a commutative diagram of
probability spaces as input, modify it, but preserve important
properties of the diagram. The precise results are formulated as
Theorems \ref{p:contraction} and \ref{p:expansion-simple} in the main
text. Their formulation requires language, notation and definitions
that we review in Section \ref{se:preliminaries}.

However, to give an idea of the results in this paper, we now present
two examples.

\subsection{Two examples.}
\subsubsection{Arrow contraction and expansion in a two-fan}
Suppose we are given a fan $\Zcal=(X\ot Z\to Y)$ and we would like to
complete it to a diamond 
\[\tageq{xyzv-diamond}
\Zcal_{\diamond}=\left(
\begin{cd}[row sep=-1mm,column sep=small]
  \mbox{}
  \&
  Z
  \arrow{dl}
  \arrow{dr}
  \&
  \mbox{}
  \\
  X
  \arrow{dr}
  \&\&
  Y
  \arrow{dl}
  \\
  \mbox{}
  \&
  V
  \&
  \mbox{}
\end{cd}
\right)
\]
such that the entropy of $V$, denoted by $[V]$, equals the mutual
information $[X:Y]$ between $X$ and $Y$. That is, we would like to
realize the mutual information between $X$ and $U$ by a pair of
reductions $X\to V$ and $U\to V$. This is not always possible, not
even approximately.

Arrow contraction instead produces another fan $\Zcal'=(X\ot Z'\to
V)$, such that the reduction $Z'\to X$ is an isomorphism and the
relative entropy $[X\rel V]$ of $X$ given $V$ equals $[X\rel U]$. By
collapsing this reduction we obtain as a diagram just the reduction
$X\to V$. If need be, we can still keep the original spaces $Z$ and
$U$ in the modified diagram obtaining the ``broken diamond'' diagram
\[
\begin{cd}[row sep=0mm,column sep=tiny]
  \mbox{}
  \&
  Z
  \arrow{dl}
  \arrow{dr}
  \&
  \mbox{}
  \\
  X
  \arrow{dr}
  \&\&
  Y
  \\
  \mbox{}
  \&
  V
  \&
  \mbox{}
\end{cd} 
\]
such that $[V]=[X:Y]$. Of course, no special technique is necessary to
achieve this result, since it is easy to find a reduction from a
tropical space $[X]$ to another tropical probability space with the
prespecified entropy, as long as the Shannon inequality is not
violated.

However, a similar operation becomes non-trivial and in fact
impossible without passing to the tropical limit, if instead of a
single space $X$ there is a more complex sub-diagram as in the example
in the next subsection.

To explain how arrow expansion works, lets start with the chain of
reductions $Z\to X\to V$. Can we extend it to a diamond, as
in~(\ref{eq:xyzv-diamond}) so that $[X:Y\rel V]=0$? This is again not
possible, in general. However, if we pass to tropical diagrams, then
such an extension always exists.

\subsubsection{One More Example of Arrow Expansion and Contraction}
Consider a diagram presented in
Figure~\ref{fig:contraction-in-L3}. Such a diagram is called a
$\Lambdabf_{3}$-diagram. We would like to find a reduction $X\to V$ so
that $[\Xcal\rel U]=[\Xcal\rel V]$. It is not possible to achieve this
within the realm of diagrams of classical probability spaces. But once
we pass to the tropical limit, the reduction $[X]\to{[V]}$ can be
found by contracting and then collapsing the arrow $[Z]\to{} [X]$, as
shown in Figure~\ref{fig:contraction-in-L3}.

\begin{figure}
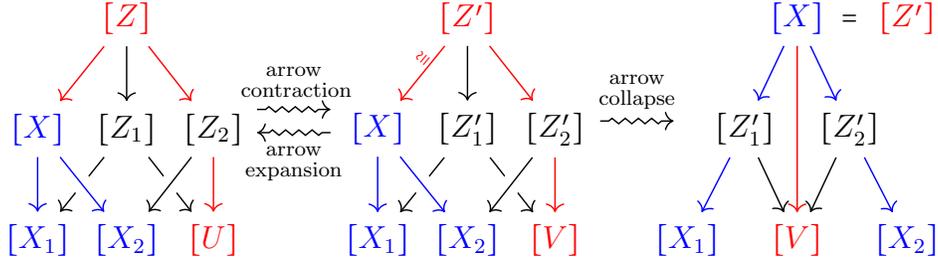

  \begin{tabular}{ccccc}
    \begin{cd}[column sep=-0.1mm,row sep=normal]
      \mbox{}
      \&{}
      \textcolor{red}{[Z]}
      \arrow[color=red]{dl}
      \arrow{d}
      \arrow[color=red]{dr}
      \&{}
      \\{}
      \textcolor{blue}{[X]}
      \arrow[color=blue]{d}
      \&{}
      [Z_{1}]
      \arrow{dr}
      \arrow{dl}
      \&{}
      [Z_{2}]
      \arrow[color=red]{d}
      \\{}
      \textcolor{blue}{[X_{1}]}
      \&{}
      \textcolor{blue}{[X_{2}]}
      \arrow[from=ul, color=blue,crossing over]
      \arrow[from=ur, crossing over]
      \&{}
      \textcolor{red}{[U]}
    \end{cd}      
    &
    \hspace{-2.05em}
    \begin{cd}[row sep=0mm]
      \mbox{}
      \arrow[rightsquigarrow]{r}{%
	\parbox{14mm}{\tiny\center arrow\\[-0.6ex] contraction\\[0.5ex]}}
      \&
      \mbox{}
      \\
      \mbox{}
      \&
      \mbox{}
      \arrow[rightsquigarrow]{l}{%
	\parbox{14mm}{\tiny\center arrow\\[-0.6ex] expansion\\[0.5ex]}}
    \end{cd}
    \hspace{-2.01em}
    &
    \begin{cd}[column sep=0mm]
      \mbox{}
      \&{}
      \textcolor{red}{[Z']}
      \arrow[color=red]{dl}[above]{\rotatebox{50}{$\cong$}}
      \arrow{d}
      \arrow[color=red]{dr}
      \&{}
      \\{}
      \textcolor{blue}{[X]}
      \arrow[color=blue]{d}
      \&{}
      [Z'_{1}]
      \arrow{dr}
      \arrow{dl}
      \&{}
      [Z'_{2}]
      \arrow[color=red]{d}
      \\{}
      \textcolor{blue}{[X_{1}]}
      \&{}
      \textcolor{blue}{[X_{2}]}
      \arrow[from=ul, color=blue,crossing over]
      \arrow[from=ur, crossing over]
      \&{}
      \textcolor{red}{[V]}
    \end{cd}      
    &
    \hspace{-2em}
    \begin{cd}
      \mbox{}
      \arrow[rightsquigarrow]{r}{%
	\parbox{14mm}{\tiny\center arrow\\[-0.6ex] collapse\\[0.5ex]}}
      \&
      \mbox{}
    \end{cd}
    \hspace{-1em}
    &
    \hspace{-10mm}
    \begin{cd}[column sep=-4mm]
      \mbox{}
      \&{}
      \&{}
      \textcolor{blue}{[X]}
      \arrow[color=blue]{dl}
      \arrow[color=red]{dd}
      \arrow[color=blue]{dr}
      \&
      =
      \&{}
      \textcolor{red}{[Z']}
      \\{}
      \&{}
      [Z'_{1}]
      \arrow[color=blue]{dl}
      \arrow{dr}
      \&{}
      \&{}
      [Z'_{2}]
      \arrow[color=blue]{dr}
      \arrow{dl}
      \&{}
      \\{}
      \textcolor{blue}{[X_{1}]}
      \&{}
      \&{}
      \textcolor{red}{[V]}
      \&{}
      \&{}
      \textcolor{blue}{[X_{2}]}
    \end{cd}         
  \end{tabular}
  \caption{Arrow contraction and expansion in a
    $\Lambdabf_{3}$-diagram. \emph{The fan $([X]\ot{}[Z]\to{}[U])$ is
      admissible. Spaces $[Z_{1}]$, $[Z_{2}]$ and $[Z]$ belong to the
      co-ideal $\lf U\rf$.}}
  \label{fig:contraction-in-L3}
\end{figure}

\medskip
Arrow contraction is closely related to the Shannon channel coding
theorem. This is perhaps most obvious from the proof. Furthermore,
arrow contraction has connections with rate regions as introduced by
Ahlswede and K\"orner, see \cite{Ahlswede-Connection-1977,
  Ahlswede-Common-2006}. These results by Ahlswede and K\"orner were
applied by \cite{Makarychev-New-2002}, resulting in a new non-Shannon
information inequality. Moreover, in \cite{Makarychev-New-2002} a new
proof was given of the results; this new proof is similar to the proof
of the arrow contraction result in the present paper.
	
The main contribution of our work lies in the fact that we prove a
much stronger preservation of properties of the diagram under arrow
contraction.

%% file: section-preliminaries.tex
\subsection{Probability spaces and their diagrams}
Our main objects of study will be \term{commutative diagrams of
  probability spaces}.  A \term{finite probability space} $X$ is a set
with a probability measure on it, supported on a finite set. We denote
by $|X|$ the cardinality of the support of the measure. The statement
$x\in X$ means that point $x$ is an atom with positive weight in $X$.
For details see
\cite{Matveev-Asymptotic-2018,Matveev-Tropical-2019,Matveev-Conditioning-2019}.

Examples of commutative diagrams of probability spaces are shown in
Figure~\ref{fig:diagram-examples}.  The objects in such diagrams are
finite probability spaces and morphisms are equivalence classes of
measure-preserving maps. Two such maps are considered to be
equivalent, if they coincide on a set of full measure. To record the
combinatorial structure of a commutative diagram, i.e. the arrangement
of spaces and morphisms, we use \term{indexing categories}, which are
poset categories satisfying an additional property, that we describe
below.

\begin{figure}
\begin{tabular}{ccc}
\begin{cd}[column sep=small]
  \mbox{}
  \&
  Z
  \arrow{dl}
  \arrow{dr}
  \&
  {}
  \\
  X
  \&\mbox{}\&
  Y
\end{cd}
&
\begin{cd}[row sep=3mm,column sep=small]
  \mbox{}
  \&
  Z
  \arrow{dl}
  \arrow{dr}
  \&
  \mbox{}
  \\
  X
  \arrow{dr}
  \&\mbox{}\&
  Y
  \arrow{dl}
  \\
  \mbox{}
  \&
  U
\end{cd}
&
\begin{cd}[row sep=3mm,column sep=small]
  \mbox{}
  \&
  T
  \arrow{dl}
  \arrow{d}
  \arrow{dr}
  \&
  \mbox{}
  \\
  U
  \arrow{d}
  \&
  V
  \arrow{dl}
  \arrow{dr}
  \&
  W
  \arrow{d}
  \\
  X
  \&
  Y
  \arrow[from=ul,crossing over]{}
  \arrow[from=ur,crossing over]{}
  \&
  Z
  \\
\end{cd}
\\
1. A fan
&
2. A diamond diagram
&
3. Full diagram on 3 spaces
\end{tabular}
\caption{Examples of diagrams of probability spaces}
\label{fig:diagram-examples}
\end{figure}

\subsubsection{Indexing Categories} 
A \term{poset category} is a finite category such that there are at
most one morphism between any two objects either way.

For a pair of objects $k,l$ in a poset category
$\Gbf=\set{i;\,\gamma_{ij}}$, such that there is a morphism
$\gamma_{kl}$ in $\Gbf$, we call $k$ an ancestor of $l$ and $l$
a descendant of $k$. The set of all ancestors of an object $k$
together with all the morphisms between them is itself a poset
category and will be called a \term{co-ideal} generated by $k$ and
denoted by $\lf k\rf$. Similarly, a poset category consisting of all
descendants of $k\in\Gbf$ and morphisms between them will be called an
\term{ideal} generated by $k$ and denoted $\lc k\rc$.

An \term{indexing category} $\Gbf=\set{i;\,\gamma_{ij}}$ used for
indexing diagrams, is a poset category satisfying the following
additional property: for any pair of objects $i_{1},i_{2}\in\Gbf$ the
intersection of co-ideals is also a co-ideal generated by some object
$i_{3}\in\Gbf$,
\[
  \lf i_1 \rf \cap \lf i_{2} \rf = \lf i_3 \rf
\]
In other words, for any pair of objects $i_{1},i_{2}\in\Gbf$ there
exists a \term{least common ancestor} $i_{3}$, that is $i_{3}$ is an
ancestor to both $i_{1}$ and $i_{2}$ and any other common ancestor is
also an ancestor of $i_3$.  Any indexing category is \term{initial},
i.e. there is a (necessarily unique) \term{initial} object
$\hat\imath$ in it, which is the ancestor of any other object in
$\Gbf$, in other words $\Gbf=\lc \hat\imath\rc$.

A \term{fan} in a category is a pair of morphisms with the same
domain. A fan $(i\ot k\to j)$ is called \term{minimal}, if for any
other fan $(i\ot l\to j)$ included in a commutative diagram
\[
\begin{cd}[row sep=-1mm]
  \mbox{}\&
  k
  \arrow{dl}
  \arrow{dr}
  \arrow{dd}
  \\
  i
  \&\&
  j
  \\
  \&
  l
  \arrow{ul}
  \arrow{ur}
\end{cd}
\]
the vertical morphism $(k\to l)$ must be an isomorphism.
Any indexing category also satisfies the property that for any pair of
objects in it there exists a minimal fan with target objects the given ones.

This terminology will also be applied to diagrams of probability
spaces indexed by $\Gbf$. Thus, given a space $X$ in a $\Gbf$-diagram,
we can talk about its ancestors, descendants, co-ideal $\lf X\rf$ and
ideal $\lc X\rc$. We use square brackets to denote tropical
  diagrams and spaces in them. For the (co-)ideals in tropical diagrams, in order
to unclutter notations, we will write
\[
  \lf X\rf:=\lf [X]\rf
  \quad\text{and}\quad
  \lc X\rc:=\lc [X]\rc
\]

\subsubsection{Diagrams}
For an indexing category $\Gbf=\set{i;\,\gamma_{ij}}$ and a category
$\mathbf{Cat}$, a commutative $\Gbf$-diagram
$\Xcal=\set{X_{i};\,\chi_{ij}}$ is a functor
$\Xcal:\Gbf\to\mathbf{Cat}$. A diagram $\Xcal$ is called \term{minimal} if it
maps minimal fans in $\Gbf$ to minimal fans in $\mathbf{Cat}$.

A \term{constant $\Gbf$-diagram} denoted $X^{\Gbf}$ is a diagram where all
the objects equal to $X$ and all morphisms are identities.

Important examples of indexing categories are a two-fan, a diamond
category, a full category $\Lambdabf_{n}$ on $n$ spaces, chains
$\Cbf_{n}$. For detailed description and more examples, the reader is
referred to the articles cited at the beginning of this section.

\subsection{Tropical Diagrams}

\subsubsection{Intrinsic Entropy Distance}
For a fixed indexing category $\Gbf$ the space of commutative
$\Gbf$-diagrams will be denoted by $\prob\<\Gbf\>$.
Evaluating entropy on every space in a $\Gbf$-diagram gives a map
\[
\ent_{*}:\prob\<\Gbf\>\to \Rbb^{\Gbf}
\]
where the target space $\Rbb^{\Gbf}$ is the space of real-valued
functions on objects of $\Gbf$. We endow this space with the
$\ell^{1}$-norm.
For a fan $\Fcal=(\Xcal\ot \Zcal\to\Ycal)$ of $\Gbf$-diagrams we
define the entropy distance between it terminal objects by
\[     
  \kd(\Fcal)
  :=
  \left\| \ent_{*}\Zcal - \ent_{*}\Xcal\right\|_{1}
  +
  \left\| \ent_{*}\Zcal - \ent_{*}\Ycal\right\|_{1}
\]     
and the intrinsic entropy distance between two arbitrary $\Gbf$-diagrams
by
\[
  \ikd(\Xcal,\Ycal)
  :=
  \inf\set{\kd(\Fcal)\st \Fcal=(\Xcal\ot \Zcal\to\Ycal)}
\]
The triangle inequality for $\ikd$ and various its other properties are
discussed in~\cite{Matveev-Asymptotic-2018}.

In the same article, also a useful estimate for the intrinsic entropy distance is proven, called the Slicing Lemma. The
following corollary, \cite[Corollary 3.10(1)]{Matveev-Asymptotic-2018},
of the Slicing Lemma will be used in the next section.  

\begin{proposition}{p:slicing}
  Let $\Gbf$ be an indexing category, $\Xcal,\Ycal\in\prob\<\Gbf\>$ and
  $U\in\prob$ included into a pair of two-fans
  \[
  \begin{cd}[row sep=-2mm]
    \mbox{}
    \&
    \tilde\Xcal
    \arrow{dl}
    \arrow{dr}
    \mbox{}
    \\
    \Xcal
    \&\mbox{}\&
    U^{\Gbf}
  \end{cd}
  \quad\quad\quad
  \begin{cd}[row sep=-2mm]
    \mbox{}
    \&
    \tilde\Ycal
    \arrow{dl}
    \arrow{dr}
    \&
    \mbox{}
    \\
    U^{\Gbf}
    \&\mbox{}\&
    \Ycal
  \end{cd}
  \]
  Then
  \[
    \ikd(\Xcal,\Ycal) 
    \leq 
    \int_{U}\ikd(\Xcal\rel u,\Ycal\rel u)\d p_{U}(u)
    +2\cdot\size{\Gbf}\cdot\ent(U)
  \]
  \end{proposition}
 
\subsubsection{Tropical Diagrams} 
Points in the asymptotic cone of $(\prob\<\Gbf\>,\ikd)$ are called
tropical $\Gbf$-diagrams and the space of all tropical
$\Gbf$-diagrams, denoted $\prob[\Gbf]$, is endowed with the
\term{asymptotic entropy distance}. We explain this now in more detail and a more extensive description can be found in \cite{Matveev-Tropical-2019}. 

To describe points in
$\prob[\Gbf]$ we consider certain sequences
$\bar\Xcal:=(\Xcal(n):\,n\in \Nbb)$ of $\Gbf$-diagrams, that grow
almost linearly and endow the space of all such sequences with the
\term{asymptotic entropy distance} defined by
\[
  \aikd(\bar\Xcal,\bar\Ycal)
  :=
  \lim_{n\to\infty}\frac1n \ikd\big(\Xcal(n),\Ycal(n)\big)
\]
A tropical diagram $[\Xcal]$ is defined to be an equivalence class of
such sequences, where two sequences $\bar\Xcal$ and $\bar\Ycal$ are
equivalent, if $\aikd(\bar\Xcal,\bar\Ycal)=0$. The space $\prob[\Gbf]$
carries the asymptotic entropy distance and
has the structure of a $\Rbb_{\geq0}$-semi-module -- one can take linear combinations with non-negative coefficients of tropical diagrams.
The linear entropy functional $\ent_{*}:\prob[\Gbf]\to\Rbb^{\Gbf}$ is
defined by 
\[
\ent_{*}[\Xcal]:=\lim_{n \to \infty} \frac1n \ent_{*}\Xcal(n)
\]

A detailed discussion about
tropical diagrams can be found in~\cite{Matveev-Tropical-2019}.
In the cited article we show that, the space $\prob[\Gbf]$ is
metrically complete, and isometrically isomorphic to a closed convex
cone in some Banach space.

For $\Gbf=\Cbf_{k}$ a \term{chain category}, containing $k$ objects
$\set{1,\dots,k}$ and unique morphism $i\to j$ for every
pair $i\geq j$, we have shown that the space $\prob[\Cbf_{k}]$ is isomorphic to
the following cone in $(\Rbb^{k},\left\| \,\cdot\, \right\|_{1})$
\[
  \prob[\Cbf_{k}]
  \cong
  \set{
    \left(\begin{array}{c}\!\!x_{1}\!\!\\\vdots\\\!\!x_{k}\!\!\end{array}\right)
    \st
    0\leq x_{1}\leq\dots\leq x_{k}
  }
\]
The isomorphism is given by the entropy functional.
Thus we can identify tropical probability spaces (elements in $\prob[\Cbf_{1}]$)
with non-negative numbers via entropy. We will simply write $[X]$ to
mean the entropy of the space $[X]$. Along this line we also adopt the
notations $[X\rel Y]$, $[X:Y]$ and $[X:Y\rel Z]$ for the relative
entropy and mutual information for the tropical spaces included in some
diagram.

\subsection{Asymptotic Equipartition Property for Diagrams}
\subsubsection{Homogeneous diagrams}

A $\Gbf$-diagram $\Xcal$ is called \term{homogeneous} if the
automorphism group $\Aut(\Xcal)$ acts transitively on every space in
$\Xcal$. Homogeneous probability spaces are uniform. For more complex
indexing categories this simple description is not sufficient.

\subsubsection{Tropical Homogeneous Diagrams}
The subcategory of all homogeneous $\Gbf$-diagrams will be denoted
$\Prob\<\Gbf\>_{\hsf}$ and we write $\Prob\<\Gbf\>_{\hsf,\msf}$ for
the category of minimal homogeneous $\Gbf$-diagrams. These spaces are
invariant under the tensor product, thus they are metric Abelian
monoids.

Passing to the tropical limit we obtain spaces of tropical (minimal)
homogeneous diagrams, that we denote $\Prob[\Gbf]_{\hsf}$ and
$\Prob[\Gbf]_{\hsf,\msf}$.

\subsubsection{Asymptotic Equipartition Property}
In~\cite{Matveev-Asymptotic-2018} the following theorem is proven
\begin{theorem}{p:aep-complete}
  Suppose $\Xcal\in\prob\<\Gbf\>$ is a $\Gbf$-diagram of
  probability spaces for some fixed indexing category $\Gbf$.
  Then there exists a sequence
  $\bar\Hcal=(\Hcal_{n})_{n=0}^{\infty}$ of homogeneous
  $\Gbf$-diagrams such
  that
  \[\tageq{quantaep} 
    \frac{1}{n} 
    \ikd (\Xcal^{n},\Hcal_{n}) 
    \leq 
    C(|X_0|,\size{\Gbf}) \cdot \sqrt{\frac{\ln^3 n}{n}} 
  \]
  where $C(|X_0|, \size{\Gbf})$ is a constant only depending on $|X_0|$
  and $\size{\Gbf}$.
\end{theorem}

The approximating sequence of homogeneous diagrams is evidently
quasi-linear with the defect bounded by the admissible function
\[
  \phi(t)
  :=
   2C(|X_0|,\size{\Gbf})\cdot t^{3/4}\geq 2C(|X_0|,\size{\Gbf})\cdot t^{1/2}\cdot \ln^{3/2}t
\]
Thus, Theorem~\ref{p:aep-complete} above states that
$\lin(\prob\<\Gbf\>)\subset\prob[\Gbf]_{\hsf}$. On the other hand we
have shown in~\cite{Matveev-Tropical-2019}, that the space of linear
sequences $\lin(\prob\<\Gbf\>)$ is dense in $\prob[\Gbf]$. Combining
the two statements we get the following theorem.

\begin{theorem}{p:aep-tropical} For any indexing category $\Gbf$, the space 
  $\prob[\Gbf]_{\hsf}$ is dense in $\prob[\Gbf]$. Similarly, the space
  $\prob[\Gbf]_{\hsf,\msf}$ is dense in $\prob[\Gbf]_{\msf}$. 
\end{theorem}

It is possible that the spaces $\prob[\Gbf]_{\hsf}$ and $\prob[\Gbf]$
coincide. At this time we have neither a proof nor a counterexample to
this conjecture.

\subsection{Conditioning in Tropical Diagrams}
For a tropical $\Gbf$-diagram $[\Xcal]$ containing a space $[U]$ we
defined a conditioned diagram $[\Xcal\rel U]$. It can be understood
as the tropical limit of the sequence $(\Xcal(n)\rel u_{n})$, where
$(\Xcal(n))$ is the homogeneous approximation of $[\Xcal]$, $U(n)$ is
the space in $\Xcal(n)$ that corresponds to $[U]$ under combinatorial
isomorphism and $u_{n}$ is any atom in $U(n)$.

We have shown in \cite{Matveev-Conditioning-2019} that operation of
conditioning is Lipschitz-continuous with respect to the asymptotic
entropy distance.

%% file: section-contraction.tex
\subsection{Arrow Collapse, Arrow Contraction and Arrow Expansion}
\subsubsection{Prime Morphisms}
A morphism $\gamma_{ij}:i\to j$ in an indexing category
$\Gbf=\set{i;\gamma_{ij}}$ will be called \term{prime} if it cannot be
factored into a composition of two non-identity morphisms
in$\Gbf$. Morphism in a $\Gbf$-diagram indexed by a prime morphism in
$\Gbf$ will also be called \term{prime}.

\subsubsection{Arrow Collapse}
Suppose $\Zcal$ is a $\Gbf$-diagram such that for some pair
$i,j\in\Gbf$, the prime morphism $\zeta_{ij}:Z_{i}\to Z_{j}$ is an
isomorphism.  \term{Arrow collapse} applied to $\Zcal$ results
  in a new diagram $\Zcal'$ obtained from $\Zcal$ by
identifying $Z_{i}$ and $Z_{j}$ via the isomorphism $\zeta_{ij}$. The
combinatorial type of $\Zcal'$ is different from that of $\Zcal$. The
spaces $Z_{i}$ and $Z_{j}$ are replaced by a single space and the new
space will inherit all the morphisms in $\Zcal$ with targets and
domains $Z_{i}$ and $Z_{j}$.

\subsubsection{Arrow Contraction and Expansion}
\term{Arrow contraction and expansion} are two operations on tropical
$\Gbf$-diagrams. Roughly speaking, arrow contraction applied to a tropical
$\Gbf$-diagram $[\Zcal]$ results in another tropical
$\Gbf$-diagram $[\Zcal']$ such that one of the arrows become an
isomorphism, while some parts of the diagram are not modified. Arrow
expansion is an inverse operation to arrow contraction. 

\subsubsection{Admissible and Reduced Sub-fans}
An \term{admissible fan} in a $\Gbf$-diagram $\Zcal$ is a minimal fan
$X\ot Z\to U$, such that $Z$ is the initial space of $\Zcal$ and any
space in $\Zcal$ belongs either to the co-ideal $\lc X\rc$ or ideal
$\lf U\rf$.

An admissible fan will be called \term{reduced} if the morphism $Z\to X$ is
an isomorphism.

\subsection{The Contraction Theorem}
Our aim is to prove the following theorem.
\begin{theorem}{p:contraction}
	Let $([X]\ot{}[Z]\to{}[U])$ be an admissible fan in some tropical
	$\Gbf$-diagram $[\Zcal]$. Then for every
	$\epsilon>0$ there exists a $\Gbf$-diagram $[\Zcal']$ containing
	an admissible fan
	$([X']\ot{}[Z']\to{}[U'])$, corresponding to the original admissible fan through the combinatorial isomorphism, such that, with the notations $\Xcal=\lc
	X\rc$ and $\Xcal'=\lc X'\rc$, the diagram $[\Zcal']$ satisfies
	\begin{enumerate}\def\theenumi{\roman{enumi}}
		\item 
		$\aikd([\Xcal'\rel U'],[\Xcal\rel U])\leq\epsilon$
		\item
		$\aikd(\Xcal',\Xcal)\leq \epsilon$
		\item
		$[Z'\rel X']\leq \epsilon$
	\end{enumerate}
\end{theorem}

It is not clear that constructing diagrams $\Zcal'$ as in the theorem
above for a sequence of values of parameter $\epsilon$ decreasing to
$0$, we can obtain a convergent sequence in $\prob[\Gbf]$ with the
limiting diagram satisfying conclusions of the theorem with
$\epsilon=0$.

The proof of Theorem~\ref{p:contraction} is based on the following
proposition, which will be proven in Section~\ref{s:proof}.

\begin{proposition}{p:contraction-approximation}
  Let $(X_{0}\ot{} Z_{0}\to{} U)$ be an admissible fan in some
  \emph{homogeneous} $\Gbf$-diagram of probability spaces $\Zcal$.
  Then there exists $\Gbf$-diagram $\Zcal'$ containing the admissible fan $(X'_{0}\ot{} Z'_{0}\to{}U')$ such that, with the
  notations $\Xcal:=\lc X_{0}\rc$ and $\Xcal':=\lc X'_{0}\rc$, it
  holds that
  \begin{enumerate}
  \item
    $\Xcal\rel u=\Xcal' \rel u'$
    for any $u\in U$ and $u'\in U'$.
  \item
    $\aikd(\Xcal, \Xcal') \leq \ikd(\Xcal,\Xcal')\leq 20\cdot\size{\Gbf}$
  \item
    $[Z'_{0}\rel X'_{0}] \leq 4\ln\ln|X_{0}|$
  \end{enumerate}
\end{proposition}
\begin{proof}[of Theorem \ref{p:contraction}]
First we assume that $[\Zcal]$ is a homogeneous tropical
diagram. It means that it can be represented by a quasi-linear
sequence $(\Zcal(n))_{n\in\Nbb_{0}}$ of homogeneous diagrams, with
defect of the sequence bounded by the function $\phi(t):=C\cdot
t^{3/4}$ for some $C>0$. This means that for any $m,n\in\Nbb$
\begin{align*}
  \aikd(\Zcal(m)\otimes\Zcal(n),\Zcal(m+n))&\leq\phi(m+n)\\
  \aikd(\Zcal^{m}(n),\Zcal(m\cdot n))&\leq D_{\phi}\cdot m\cdot\phi(n)
\end{align*}
where $D_{\phi}$ is some constant depending on $\phi$, see \cite{Matveev-Tropical-2019}.

Fix a number $n\in\Nbb$ and apply
Proposition~\ref{p:contraction-approximation} to the
\emph{homogeneous} diagram $\Zcal(n)$, containing the admissible
fan $X_{0}(n)\ot Z_{0}(n)\to U(n)$ and sub-diagram
$\Xcal(n)=\lc X_{0}(n)\rc$. As a result we obtain a diagram $\Zcal''$
containing the fan $X''_{0}\ot Z''_{0}\to U''$ and the sub-diagram
$\Xcal''=\lc X''_{0}\rc$, such that
\begin{align*}
  &\Xcal''\rel u''
  =
  \Xcal(n)\rel  u
  &\text{\hspace{-5em}for any $u''\in U''$ and $u\in U(n)$}
  \\\tageq{contraction-Z(n)}
  &\aikd(\Xcal'',\Xcal(n))
  \leq 
  20\size{\Gbf}
  \\
  &[Z''_{0}\rel X''_{0}]
  \leq 
  4\ln\ln|X_{0}(n)|
\end{align*}

Define the two tropical diagrams 
\begin{align*}
  [\Zcal']&:=\frac{1}{n}\bernoulli{\Zcal''}\\
  [\tilde\Zcal]&:=\frac{1}{n}\bernoulli{\Zcal(n)}
\end{align*}

Since $\Xcal''\rel u''$ does not depend on $u''$ and $\Xcal(n)\rel u$
does not depend on $u$ we have $[\Xcal'\rel U']=(1/n)\cdot \bernoulli{(\Xcal''\rel
  u'')}$ and $[\tilde\Xcal\rel \tilde U]=(1/n)\cdot\bernoulli{(\Xcal(n)\rel u)}$. From~(\ref{eq:contraction-Z(n)}), we obtain
\begin{align*}
  &[\Xcal'\rel U']
  =
  [\tilde\Xcal\rel\tilde U]
  \\\tageq{contraction-bernoulli-Z(n)}
  &\aikd\big([\Xcal'],[\tilde\Xcal]\big)
  \leq 
  \frac{20\size{\Gbf}}{n}
  \\
  &[Z'_{0}\rel X'_{0}]
  \leq 
  \frac{4\ln\ln|X_{0}(n)|}{n}
\end{align*}
The distance between $[\tilde\Zcal]$ and $[\Zcal]$ can be bounded as follows
\begin{align*}\tageq{distance-Z-Ztilde}
  \aikd\big([\tilde\Zcal],[\Zcal]\big)
  &=
  \frac{1}{n}\aikd\big(\bernoulli{\Zcal(n)},n\cdot[\Zcal]\big)
  =
  \frac{1}{n}\lim_{m\to\infty}
  \frac1m \aikd\big(\Zcal^{m}(n),\Zcal(m\cdot n)\big)
  \\
  &\leq
  \frac{1}{n}D_{\phi}\cdot\phi(n) 
\end{align*}
This also implies
\[\tageq{distance-X-Xtilde}
  \aikd\big([\tilde\Xcal],[\Xcal]\big)
  \leq
  \frac{1}{n}D_{\phi}\cdot\phi(n)
\]
Since conditioning is a Lipschitz continuous operation with Lipschitz
constant $2$, we also have
\[\tageq{distance-conditioning}
   \aikd\big([\tilde\Xcal\rel \tilde U],[\Xcal\rel U]\big)
   \leq
   \frac{2}{n}D_{\phi}\cdot\phi(n) 
\]

Combining the estimates
in~(\ref{eq:contraction-bernoulli-Z(n)}),%
~(\ref{eq:distance-Z-Ztilde}),%
~(\ref{eq:distance-X-Xtilde})
and~(\ref{eq:distance-conditioning}) we obtain
\begin{align*}
  &\aikd\big([\Xcal'\rel U'],[\Xcal\rel U]\big)
  \leq
  2D_{\phi}\cdot\frac{\phi(n)}{n}
  \\
  &\aikd\big([\Xcal'],[\Xcal]\big)
  \leq 
  \frac{20\size{\Gbf}}{n} + D_{\phi}\frac{\phi(n)}{n}
  \\
  &[Z'_{0}\rel X'_{0}]
  \leq 
  \frac{4\ln\ln|X_{0}(n)|}{n} + 2D_{\phi}\frac{\phi(n)}{n}
\end{align*}

Note that $|X_{0}(n)|$ grows at most exponentially (it is bounded by
$\ebf^{n([X_{0}]+C)}$ for some $C$) and $\phi$ is a strictly sub-linear
function. Thus by choosing $n$ sufficiently large depending on given
$\epsilon>0$ we obtain $[\Zcal']$ satisfying conclusions of the
theorem for $[\Zcal]$ homogeneous.

To prove the theorem in full generality observe that all the
quantities on the right-hand side of the inequalities are
Lipschitz-continuous. Since $\prob[\Gbf]_{\hsf}$ is dense in
$\prob[\Gbf]$ the theorem extends to any $[\Zcal]$ by first
approximating it with any precision by a homogeneous configuration and
applying the argument above.
\end{proof}

\subsection{The expansion Theorem}
The following theorem is complimentary to
Theorem~\ref{p:contraction}. The expansion applied to a diagram
containing a reduced admissible fan produces a diagram with an
admissible fan, such that contraction of it is the original diagram.
Thus, arrow expansion is a right inverse of the arrow contraction
operation.

In general, contraction erases some information stored in the diagram,
so there are many right inverses. We prove the theorem below by
providing a simple construction of one such right inverse.

\begin{theorem}{p:expansion-simple}
  Let $([X]\ot{}[Z']\to{}[U'])$ be a reduced admissible fan in some
  tropical $\Gbf$-diagram $[\Zcal']$ and $\lambda>0$. Let
  $[\Xcal]:=\lc X\rc$. Then there exists $\Gbf$-diagram $[\Zcal]$
  containing the copy of $[\Xcal]$, such that the corresponding
  admissible fan $([X]\ot{}[Z]\to{}[U])$ has $[Z\rel X]=\lambda$ and
  $[\Xcal\rel U]=[\Xcal\rel U']$.
\end{theorem}

\begin{proof}
Let $[W]$ be a tropical probability space with entropy equal to
$\lambda$. For any reduction of tropical spaces $[A]\to{}[B]$, there are natural
reductions
\begin{align*}
  \big([A]+[W]\big)&\to \big([B]+[W])\\
  \big([A]+[W]\big)&\to{} [W]\\
\end{align*}

We construct the diagram $[\Zcal]$ by replacing every space $[V]$ in
the ideal $\lf U'\rf$ with $[U]+[W]$. Every morphisms
$[V_{1}]\to{}[V_{2}]$ within $\lf U'\rf$ is replaced by
\[
\big([V_{1}]+[W]\big)\to \big([V_{2}]+[W])
\]
And any morphism from $[V]$ in $\lf U'\rf$ to a space $[Y]$ in $\lc
X\rc$ is replaced by a composition
\[
\big([V]+[W]\big)\to{} [V]\to{}[Y]
\]
Clearly the resulting diagram satisfies the conclusion of the theorem.
\end{proof}

\medskip

The rest of the article is devoted to the development of necessary
tools and the proof of Proposition~\ref{p:contraction-approximation}.

%% file: section-local.tex
In this section we derive a bound, very similar to Fano's inequality,
on the intrinsic entropic distance between two diagrams of probability
spaces with the same underlying diagram of sets. The bound will be in
terms of total variation distance between two distributions
corresponding to the diagrams of probability spaces. It will be used
in the next section, to prove arrow contraction theorem.

\subsection{Distributions}
\subsubsection{Distributions on sets.}
For a finite set $S$ we denote by $\Delta S$ the collection of all
probability distributions on $S$ and by $\|\pi_{1}-\pi_{2}\|_1$ we denote the
total variation distance between $\pi_{1},\pi_{2}\in\Delta S$.

\subsubsection{Distributions on Diagrams of Sets}
\label{s:disttypes-distributions-config}
Let $\Set$ denote the category of finite sets and surjective maps.
For an indexing category $\Gbf$, we denote by $\Set\<\Gbf\>$ the category
of $\Gbf$-diagrams in $\Set$.  That is, objects in $\Set\<\Gbf\>$ are
commutative diagrams of sets indexed by the category $\Gbf$, the
spaces in the such a diagram are finite sets and arrows represent
surjective maps, subject to commutativity relations.

For a diagram of sets $\Scal=\set{S_{i};\sigma_{ij}}$ we define the
\term{space of distributions
  on the diagram} $\Scal$ by
\[
\Delta\Scal
:=
\set{(\pi_{i})\in\prod_i\Delta S_{i}\st (\sigma_{ij})_{*}\pi_{i}=\pi_{j}}
\]
where $f_{*}:\Delta S\to \Delta S'$ is the affine map induced by
a surjective map $f:S\to S'$. If $S_{0}$ is the initial space of
$\Scal$, then there is an isomorphism
\begin{align*} 
  \tageq{distributions-iso}      
  \Delta S_{0}&\oto[\cong]\Delta\Scal\\
  \Delta S_{0}\ni\pi_{0}
    &\mapsto
  \set{(\sigma_{0i})_{*}\pi_{0}}\in\Delta\Scal
  \\
  \Delta S_{0} \ni \pi_{0}
  &\leftmapsto
  \set{\pi_{i}}\in\Delta
\end{align*}

Using the isomorphism~(\ref{eq:distributions-iso}) we define total
variation distance between two distributions $\pi,\pi'\in\Delta\Scal$
as 
\[
\left\| \pi-\pi' \right\|_{1}:=\left\| \pi_{0} -\pi'_{0}\right\|_{1}
\]

Given a $\Gbf$-diagram of sets $\Scal=\set{S_{i};\sigma_{ij}}$ and an
element $\pi\in\Delta\Scal$ we can construct a $\Gbf$-diagram of
probability spaces $(\Scal,\pi):=\set{(S_{i},\pi_{i});\sigma_{ij}}$.

\medskip

Below we give the estimate of the entropy distance between two
$\Gbf$-diagrams of probability spaces $(\Scal,\pi)$ and $(\Scal,\pi')$
in terms of the total variation distance $\left\| \pi-\pi' \right\|$
between distributions.

\subsection{The estimate}\label{p:ikd-local-estimate}
The upper bound on the entropy distance, that we derive below, has two
summands. One is linear in the total variation distance with the
slope proportional to the $\log$-cardinality of $S_{0}$. The second one is
super-linear in the total variation distance, but it does not depend
on $\Scal$. So we have the following interesting observation: of
course, the super-linear summand always dominates the linear one
locally. However as the cardinality of $\Scal$ becomes large it is
the linear summand that starts playing the main role. This will be
the case when we apply the bound in the next section.

For $\alpha\in[0,1]$ consider a binary probability space with the
weight of one of the atoms equal to $\alpha$
\[
\Lambda_{\alpha}
:=
\big(
\set{\square,\blacksquare};\,
p(\square)=1-\alpha,\,
p(\blacksquare)=\alpha
\big)
\]

\begin{proposition}{p:ikd-local}
  For an indexing category $\Gbf$, consider  a $\Gbf$-diagram
  of sets $\Scal=\set{S_{i},\sigma_{ij}}\in\Set\<\Gbf\>$.
  Let $\pi,\pi'\in\Delta \Scal$ be two probability
  distributions on $\Scal$.
  Denote $\Xcal:=(\Scal,\pi)$, $\Ycal:=(\Scal,\pi')$ and
  $\alpha:=\frac12\left\| \pi-\pi' \right\|_1$. Then
  \[
  \ikd(\Xcal,\Ycal)
  \leq
  2\size{\Gbf}\big(\alpha\cdot\ln|S_{0}|+\ent(\Lambda_{\alpha})\big)
  \]
\end{proposition}
\begin{proof}
  To prove the local estimate we decompose both $\pi$ and $\pi'$ into a
  convex combination of a common part $\hat \pi$ and rests $\pi^{+}$ and
  $\pi'^+$.  The coupling between the common parts gives no contribution
  to the distance, and the worst possible estimate on the other parts
  is still enough to get the bound in the lemma, by using
  Proposition~\ref{p:slicing}. 

  Let $S_{0}$ be the initial set in the diagram $\Scal$.
  We will need the following obvious rough estimate of the entropy
  distance that holds for any $\pi,\pi'\in\Delta\Scal$:
  \[\tageq{roughlocalestimate}
    \ikd(\Xcal,\Ycal)\leq 2\size{\Gbf}\cdot\ln|S_{0}|
  \]
  It can be obtained by taking a tensor product for the coupling
  between $\Xcal$ and $\Ycal$.

  Our goal now is to write $\pi$ and $\pi'$ as the convex combination of
  three other distributions $\hat \pi$, $\pi^{+}$ and $\pi'^{+}$ as in
  \begin{align*}
    \pi
    &=
    (1-\alpha)\cdot\hat \pi + \alpha\cdot \pi^{+}
    \\
    \pi'
    &=
    (1-\alpha)\cdot\hat \pi + \alpha\cdot \pi'^{+}
  \end{align*}
  with the smallest possible $\alpha\in[0,1]$.

  We could do it the following way.  Let $\pi_{0}$ and $\pi_{o}'$ be
  the distributions on $S_{0}$ that correspond to $\pi$ and $\pi'$
  under isomorphisms~(\ref{eq:distributions-iso}). Let
  $\alpha:=\frac12\left\| \pi-\pi' \right\|_{1}=\frac12\left\|
  \pi_{0}-\pi_{0}' \right\|_{1}=$. If $\alpha=1$ then the proposition
  follows from the rough estimate~(\ref{eq:roughlocalestimate}), so
  from now on we assume that $\alpha<1$.  Define three probability
  distributions $\hat \pi_{0}$, $\pi_{0}^{+}$ and $\pi'^{+}_{0}$ on $S_{0}$
  by setting for every $x\in S_{0}$
  \begin{align*}
    \hat \pi_{0}(x) 
    &:= \frac1{1-\alpha}
    \min\set{\pi_{0}(x),\pi'_{0}(x)}
    \\ 
    \pi_{0}^{+} 
    &:=
    \frac{1}{\alpha}\big(\pi_{0}-(1-\alpha)\hat \pi_{0}\big)
    \\ 
    \pi'^{+}_{0} 
    &:= 
    \frac{1}{\alpha}\big(\pi'_{0}-(1-\alpha)\hat \pi_{0}\big)
  \end{align*}
  
  Denote by $\hat \pi,\pi^{+},\pi'^{+}\in\Delta\Scal$ the distributions
  corresponding to $\hat \pi_{0},\pi_{0}^{+},\pi'^{+}_{0}\in\Delta S_{0}$
  under isomorphism~(\ref{eq:distributions-iso}). Thus we have
  \begin{align*}
    \pi&=(1-\alpha)\hat \pi+\alpha\cdot \pi^{+}\\
    \pi'&=(1-\alpha)\hat \pi+\alpha\cdot \pi'^{+}
  \end{align*}

  Now we construct two fans of
  $\Gbf$-diagrams 
  \[\tageq{fansforlocal}
  \begin{cd}[row sep=-2mm]
    \mbox{}
    \&
    \tilde\Xcal
    \arrow{dl}
    \arrow{dr}
    \mbox{}
    \\
    \Xcal
    \&\mbox{}\&
    \Lambda_{\alpha}
  \end{cd}
  \quad\quad\quad
  \begin{cd}[row sep=-2mm]
    \mbox{}
    \&
    \tilde\Ycal
    \arrow{dl}
    \arrow{dr}
    \&
    \mbox{}
    \\
    \Lambda_{\alpha}
    \&\mbox{}\&
    \Ycal
  \end{cd}
  \]
  by setting 
  \begin{align*}
    \tilde X_{i}
    &:=
    \Big(S_{i}\times\un\Lambda_{\alpha};\;
         \tilde \pi_{i}(s,\square)=(1-\alpha)\hat \pi_{i}(s),\,
         \tilde \pi_{i}(s,\blacksquare)=\alpha\cdot \pi^{+}_{i}(s)
    \Big)
    \\
    \tilde Y_{i}
    &:=
    \Big(S_{i}\times\un\Lambda_{\alpha};\;
         \tilde \pi'_{i}(s,\square)=(1-\alpha)\hat \pi_{i}(s),\,
         \tilde \pi'_{i}(s,\blacksquare)=\alpha\cdot \pi'^{+}_{i}(s)
    \Big)
    \\
  \end{align*}
  and
  \begin{align*}
    \tilde\Xcal
    &:=
    \set{\tilde X_{i};\,\sigma_{ij}\times\mathbf{id}}\\
    \tilde\Ycal
    &:=
    \set{\tilde Y_{i};\,\sigma_{ij}\times\mathbf{id}}\\
  \end{align*}
  The reductions in the fans
  in~(\ref{eq:fansforlocal}) are given by coordinate projections.  Note
  that the following isomorphisms hold
  \begin{align*}
    \Xcal\rel\square
    &\cong
    (\Scal,\hat \pi)
    \\
    \Xcal\rel\blacksquare
    &\cong
    (\Scal, \pi^{+})
    \\
    \Ycal\rel\square
    &\cong
    (\Scal, \hat \pi)
    \cong
    \Xcal\rel\square
    \\
    \Ycal\rel\blacksquare
    &\cong
    (\Scal,\pi'^{+})
  \end{align*}

  Now we apply Proposition~\ref{p:slicing} along with the rough
  estimate in~(\ref{eq:roughlocalestimate}) to obtain the desired
  inequality
  \begin{align*}
    \ikd(\Xcal,\Ycal)
    &\leq
    (1-\alpha)\ikd(\Xcal\rel\square,\Ycal\rel\square)
    +
    \alpha\cdot\ikd(\Xcal\rel\blacksquare,\Ycal\rel\blacksquare)
    \\
    &\quad
    +\sum_{i}\big[\ent(\Lambda_{\alpha}\rel X_{i})
      +\ent(\Lambda_{\alpha}\rel Y_{i})\big]
    \\
    &\leq
    2\size{\Gbf}
    \big(\alpha\cdot\ln|S_{0}|+\ent(\Lambda_{\alpha})\big)
  \end{align*}
\end{proof}

%% file: section-proof.tex
In this section we prove
Proposition~\ref{p:contraction-approximation}, which is shown below
verbatim.
\repeatclaim{p:contraction-approximation}
The proof consists of the construction in Section~\ref{s:construction}
and estimates in Propositions~\ref{p:ikd-estimate}
and~\ref{p:height-estimate}.

\subsection{The construction}\label{s:construction}

In this section we fix an indexing category $\Gbf$, a minimal $\Gbf$-diagram of
probability spaces $\Zcal$ with an admissible sub-fan $X_{0}\ot Z_{0}\to U$.
We denote $\Xcal:=\lc X_{0}\rc$ and by $\Hbf$ we denote the combinatorial
type of $\Xcal=\set{X_{i};\chi_{ij}}$. 

Instead of diagram $\Zcal$ we consider an extended diagram, which is a
two-fan of $\Hbf$-diagrams
\[\tageq{admissible-fan}
\begin{cd}[row sep=-2mm]
  \mbox{}
  \&
  \Ycal
  \arrow{dl}{\pi_{1}}
  \arrow{dr}
  \\
  \Xcal
  \&\mbox{}\&
  U^{\Hbf}
\end{cd}
\]
where $\Ycal=\set{Y_{i};\upsilon_{ij}}$ consists of those spaces in
$\Zcal$, that are initial spaces of two-fans with feet in $U$ and in
some space in $\Xcal$. That is for every $i\in\Hbf$ the space $Y_{i}$
is defined to be the initial space in the minimal fan $X_{i}\ot
Y_{i}\to U$ in $\Zcal$. It may happen that for some pair of indices
$i_{1},i_2\in\Hbf$ the initial spaces of the fans with one feet $U$
and the other $X_{i_{1}}$ and $X_{i_{2}}$ coincide in $\Zcal$. In
$\Ycal$, however, they will be treated as separate spaces, so that the
combinatorial type of $\Ycal$ is $\Hbf$.
Starting with the diagram in~(\ref{eq:admissible-fan}) one can recover
$\Zcal$ by collapsing all the isomorphism arrows. The initial space of
$\Ycal$ will be denoted $Y_{0}$.

We would like to construct a new fan $\Xcal'\ot[\pi_{1}'] \Ycal' \to
V^{\Hbf}$, such that
\[\tageq{conditions}
\begin{cases}
  \Xcal\rel u=\Xcal'\rel v & \text{for any $u\in U$ and $v\in V$}\\
  \ikd(\Xcal',\Xcal)\leq 20\size{\Gbf}\\
  [Y'_{0}\rel X'_{0}]\leq 4\ln\ln|X_{0}|
\end{cases}
\]

Once this goal is achieved, we collapse all the isomorphisms to obtain
$\Gbf$-diagram satisfying conditions in the conclusion of
Proposition~\ref{p:contraction-approximation}.

We start with a general description of the idea behind the
construction, followed by the detailed argument.  To introduce the new
space $V$ we take its points to be $N$ atoms in $u_{1},\ldots,u_{N}\in
U$. Ideally we would like to choose the atoms in such a way that
$X_{0}\rel u_{n}$ are disjoint and cover the whole of $X_{0}$. It is not
always possible to achieve this exactly. However, when $|X_{0}|$ is large,
$N$ is taken slightly larger than $\ebf^{[X_{0}:U]}$, and
$u_{1},\ldots,u_{N}$ are chosen at random, then with high probability
the spaces $X_{0}\rel u_{n}$ will overlap only little and will cover most
of $X_{0}$. The details of the construction follow.

We fix $N \in \Nbb$ and construct several new diagrams. For each of the
new diagrams we provide a verbal and formal description.
\begin{itemize}
\item 
  The space $U^{N}$. Points in it are independent samples of length
  $N$ of points in $U$.
\item
  The space $V_N=(\set{1,\dots,N},\unif)$. A point $n\in V_{N}$ should
  be interpreted as a choice of index in a sample $\bar u\in U^{N}$.
\item
  The $\Hbf$-diagram $\Acal$, where
  \begin{align*}
    \Acal
    &=
    \set{A_{i};\,\alpha_{ij}}
    \\
    A_{i}
    &=
    \big(
    \set{(x,n,\bar u)\st x\in X_i\rel u_{n}}, \unif
    \big)
    \\
    \alpha_{ij}
    &=
    (\chi_{ij},\id,\id)
  \end{align*}
  A point $(x,n,\bar u)$ in $A_{i}$ corresponds to the choice of a
  sample $\bar u\in U^{N}$, an independent choice of a member of the
  sample $u_{n}$ and a point $x\in X_{i}\rel
  u_{n}$.  Recall that the original diagram $\Zcal$ was assumed to be
  homogeneous and, in particular, the distribution on $X_{i}\rel
  u_{n}$ is uniform. Due to the assumption on homogeneity of $\Zcal$,
  the space $X_{i}\rel u$ does not depend on $u\in U$. Since $V_{N}$
  is also equipped with the uniform distribution, it follows that the
  distribution on $A_{i}$ will also be uniform.
\item 
  The $\Hbf$-diagram $\Bcal$, where
  \begin{align*}
    \Bcal
    &=
    \set{B_{i};\,\beta_{ij}}
    \\
    B_{i}
    &=
    \big(
      \set{(x,\bar u)\st x\in\bigcup_{n=1}^{N}X_{i}\rel u_{n}},
      p_{B_i}
    \big)
    \\
    \beta_{ij}
    &=
    (\chi_{ij},\id)
  \end{align*}
  A point $(x,\bar u)\in B_{i}$ is the choice of a sample $\bar u\in
  U^{N}$ and a point $x$ in one of the fibers $X_{i}\rel u_{n}$,
  $n=1,\ldots,N$. The distribution $p_{B_{i}}$ on $B_{i}$ is chosen so
  that the natural projection $A_{i}\to B_{i}$ is the reduction of
  probability spaces. Given a sample $\bar u$, if the fibers
  $X_{i}\rel u_{n}$ are not disjoint, then the distribution on
  $B_{i}\rel \bar u$ need not to be uniform. Below we will give an
  explicit description of $p_{\Bcal}$ and study the dependence of
  $p_{\Bcal}(\,\cdot\,\rel \bar u)$ on the sample $\bar u\in U^{N}$.
\end{itemize}

These diagrams can be organized into a minimal diamond diagram of
$\Hbf$-diagrams, where reductions are obvious projections.
\[
\tageq{minimized-T-diagram}
\begin{cd}[row sep=0mm]
  \&
  \Acal
  \arrow{dl}
  \arrow{dr}
  \&
  \\
  \Bcal
  \arrow{dr}
  \&\&
  \parbox{1em}{\mbox{$V_{N}\otimes U^{N}$}}
  \arrow{dl}
  \\
  \&
  U^{N}
\end{cd}
\]

To describe the probability distribution on $\Bcal$ first we define
several relevant quantities:
\begin{align*}
  \rho
  &:=
  \frac{\big|X_{0}\rel u\big|}{|X_{0}|}=\ebf^{-[X_{0}:U]}
  \\
  N(x,\bar u)
  &:=
  \big|\set{n\in V_{N} \st x\in X_{0}\rel u_{n}}\big|
  \\
  \nu(x,\bar u)
  &:=
  \frac{N(x,\bar u)}{N}
  =
  p_{V_{N}}\set{n\in V_{N} \st x\in X_{0}\rel u_{n}}
\end{align*}
Recall that the distribution $p_{\Bcal}$ is completely determined by
the distribution $p_{B_{0}}$ on the initial space of $\Bcal$ via
isomorphism~(\ref{eq:distributions-iso}). From homogeneity of $\Zcal$
it follows that distributions on both $A_{0}$ and $A\rel_{\bar u}$ are
uniform.  Therefore
\[\tageq{db-B-given-u} 
  p_{B_{0}}(x\rel\bar u)
  :=
  \frac{\nu(x,\bar u)}{\rho\cdot|X_{0}|}
\]

The desired fan $(\Xcal' \ot \Ycal' \to V^{\Hbf})$ mentioned in the beginning
of the section is obtained from the top fan in the diagram in
(\ref{eq:minimized-T-diagram}) by conditioning on $\bar{u} \in
U^N$. We will show later that for an appropriate choice of $N$ and for
most choices of $\bar{u}$, the fan we obtain in this way has the
required properties.

First, we would like to make the following observations. Fix an
arbitrary $\bar{u} \in U^N$. Then:
\begin{enumerate}
\item 
  The underlying set of the probability space $B_{0}\rel \bar
  u=X_{0}\rel\bar u$ is
  $\un X_{0}$.
\item
  The diagrams
  \begin{align*}
    \Ycal'_{\bar u}
    &:=
    \Acal\rel\bar u
    \\
    \Xcal'_{\bar u}
    &:=
    \Bcal\rel\bar u
  \end{align*}
  are included into a two-fan of $\Hbf$-diagrams
  \[
  \begin{cd}[row sep=-1mm]
    \mbox{}
    \&
    \Ycal'_{\bar u}
    \arrow{dl}
    \arrow{dr}
    \&
    \\
    \Xcal'_{\bar u}
    \&
    \&
    V_{N}
  \end{cd}
  \]
  which is obtained by conditioning top fan in the diagram
  in~(\ref{eq:minimized-T-diagram}).

  The very important observation is that diagrams $\Xcal'_{\bar u}\rel
  n$ and $\Xcal\rel u$ are isomorphic for any choice of $n\in V_{N}$
  and $u\in U$. The isomorphism is the composition of the following
  sequence of isomorphisms
  \[
  \Xcal'_{\bar u}\rel n\to
  \Bcal\rel (\bar u, {n}) \to
  \Acal\rel (\bar u, {n})\to
  \Xcal\rel u_{n}\to
  \Xcal\rel u
  \]
  where the first isomorphism follows from the definition of
  $\Xcal'_{\bar u}$,
  the second -- from minimality of the fan $\Bcal\ot\Acal\to V_{N}$,
  the third -- from the definition of $\Acal$ and the forth -- from
  the homogeneity of $\Zcal$.
\end{enumerate}

\subsection{The estimates}
We now claim and prove that one could choose a number $N$ and $\bar u$
in $U^{N}$ such that 
\begin{enumerate}
\item 
  $\ikd(\Xcal'_{\bar u},\Xcal)\leq 20\size{\Hbf}$.
\item
  $[Y'_{\bar u,0}\rel X'_{\bar u,0}]\leq 4\ln\ln|X_{0}|$, where $Y'_{\bar u,0}$
  and $X'_{\bar u,0}$ are initial spaces in $\Xcal'_{\bar u}$ and
  $\Ycal'_{\bar u}$, respectively.
\end{enumerate}

\subsubsection{Total Variation and Entropic Distance estimates}
If we fix some $x_{0}\in X_{0}$, then $\nu=\nu(x_{0},\,\cdot\,)$ is a
scaled binomially distributed random variable with parameters $N$ and
$\rho$, which means that $N\cdot\nu\sim\operatorname{Bin}(N,\rho)$.

First we state the following bounds on the tails of a binomial
distribution. 
\begin{lemma}{p:binomial-totalvar}
   Let $\nu$ be a scaled binomial random variable with parameters $N$
  and $\rho$, then 
  \begin{enumerate}\def\theenumi{\roman{enumi}}
  \item 
    for any $t \in [0,1]$ holds
    \[
      \p\set{|\nu-\rho|>\rho\cdot t}
      \leq
      2\cdot\ebf^{-\frac{1}{3}\cdot N\cdot \rho\cdot t^2}
    \]
  \item
    for any $t\in[0,2]$ holds
    \[
      \p\set{\frac{\nu}{\rho}\ln\frac{\nu}{\rho}>t}
      \leq
      \ebf^{-\frac{1}{12}\cdot N\cdot \rho\cdot t^{2}}
    \]
  \end{enumerate}
\end{lemma}
The proof of Lemma~\ref{p:binomial-totalvar} can be found at the end
of this section.

Below we use the notation $\Pbb:=p_{U^N}$ for the probability distribution
on $U^{N}$. For a pair of complete diagrams $\Ccal$, $\Ccal'$ with the same
underlying diagram of sets and with initial spaces $C_0$, $C_0'$, we
will write $\alpha(\Ccal, \Ccal')$ for the halved total variation
distance between their distributions
\[
\alpha(\Ccal,\Ccal'):=\frac12\left\|p_{C_0}-p_{C_0'}\right\|_{1}
\]

\begin{proposition}{p:totalvar-estimate}
  In the settings above, for $t \in [0,1]$, the following inequality holds
  \[
  \p\set{\bar u\in U^{N}\st 2 \alpha(\Xcal'_{\bar u},\Xcal)>t}
  \leq
  2|X_{0}|\cdot\ebf^{-\frac13 N\cdot \rho\cdot t^{2}}
  \]
\end{proposition}

\begin{proof}
Recall that by definition $\Xcal'_{\bar u}=\Bcal\rel\bar u$. We use equation (\ref{eq:db-B-given-u}) to expand the left hand side
of the inequality as follows
\begin{align*}
  \p&\set{\bar u\in U^{N}
         \st
         2 \alpha(\Bcal\rel\bar u, \Xcal)>t}
  =
  \p\set{\bar u\in U^{N}
         \st
         \sum_{x\in X_{0}}
           \left|
             \frac{\nu(x,\bar u)}{\rho\cdot|X_{0}|}
             -
             \frac{1}{|X_{0}|}
           \right|  
         > t
        }
  \\
  &=
  \p\set{\bar u\in U^{N}
         \st
         \sum_{x\in X_{0}}
           \left|
             \nu(x,\bar u)
             -
             \rho
           \right|  
         > \rho\cdot|X_{0}|\cdot t
        }
  \\
  &\leq
  \p\set{\bar u\in U^{N}
         \st
         \text{there exists $x_{0}$ such that }
           \left|
             \nu(x_{0},\bar u)
             -
             \rho
           \right|
           >\rho\cdot t
        }
  \\
  &\leq
  \sum_{x\in X_{0}}
  \p\set{\bar u\in U^{N}
         \st
           \left|
             \nu(x,\bar u)
             -
             \rho
           \right|
           >\rho\cdot t
        }
  \\
\end{align*}
Since by homogeneity of the original diagram all the summands
are the same, we can fix some $x_{0}\in X_{0}$ and estimate further:
\begin{align*}
  \p&\set{\bar u\in U^{N}
         \st
         2 \alpha(\Bcal\rel\bar u,\Xcal)>t}
  \leq
  |X_{0}|\cdot
  \p\set{\bar u\in U^{N}
         \st
                      \left|
             \nu(x_{0},\bar u)
             -
             \rho
           \right|
           > \rho\cdot t
        }
  \\
\end{align*}
Applying Lemma~\ref{p:binomial-totalvar}(i) we obtain the required
inequality.  
\end{proof}

In the propositions below we assume that $|X_0|$ is sufficiently large
(larger than $\ebf^{20}$).

\begin{proposition}{p:ikd-estimate}
  In the settings above and for any $\frac{10}{\ln|X_{0}|}\leq t \leq 1$ holds:
  \[
  \p\set{\bar u\in U^{N}:
         \ikd(\Xcal'_{\bar u},\Xcal)
         >
         t(2\cdot\size{\Gbf}\cdot\ln|X_{0}|)
        }
  \leq
  2|X_{0}|\cdot
  \ebf^{{- \frac13 N\cdot \rho\cdot t^{2}}}
  \]
\end{proposition}
\begin{proof}
  We will use local estimate to bound entropy distance and then apply
  Proposition~\ref{p:totalvar-estimate}. To simplify notations we will
  write simply $\alpha$ for $\alpha(\Xcal'_{\bar
  u},\Xcal)=\alpha(\Bcal\rel\bar
  u,\Xcal)$.  \begin{align*} \p&\set{\bar u\in
  U^{N}: \ikd(\Bcal\rel\bar u,\Xcal) >
  (2\cdot\size{\Gbf}\cdot\ln|X_{0}|)t } \\ &\leq \p\set{\bar u\in
  U^{N}: 2\cdot\size{\Gbf} (\alpha\cdot\ln|X_{0}|
  + \ent(\Lambda_{\alpha})) > (2\cdot\size{\Gbf}\cdot\ln|X_{0}|)t } \\
  &\leq \p\set{\bar u\in U^{N}: \alpha + \ent(\Lambda_{\alpha})/\ln
  |X_{0}| > t }
\end{align*}
Note that in the chosen regime, $t\geq10/\ln|X_{0}|$, the first
summand on the left-hand side of the inequality is larger than the
second, and thus
\begin{align*}
  \p&\set{\bar u\in U^{N}:
	\ikd(\Bcal\rel\bar u,\Xcal)
	>
	(2\cdot\size{\Gbf}\cdot\ln|X_{0}|)t
}  \\
   &\leq
   \p\set{\bar u\in U^{N}:
          2 \alpha > t
         }
   \\
   &\leq  
  2|X_{0}|\cdot
  \ebf^{- \frac13 N\cdot \rho\cdot t^{2}}
\end{align*}
\end{proof}

\subsubsection{The ``height'' estimate}
Recall that for given $N\in\Nbb$ and $\bar u\in U^{N}$ we have
constructed a two-fan of $\Hbf$-diagrams 
\[
  \Xcal'_{\bar u}\ot\Ycal'_{\bar u}\to V_{N}^{\Hbf}
\]
We will now estimate the length of the arrow $Y'_{\bar u,0}\to
X'_{\bar u,0}$.

\begin{proposition}{p:height-estimate}
In the settings above and for $t\in[0,2]$
\[
  \p\set{\bar u\in U^{N}
         \st
         [Y'_{\bar u,0}\rel X'_{\bar u,0}]
         > \ln(N\cdot \rho)+t
        }
  \leq
  |X_{0}|\cdot\ebf^{-\frac{1}{12}N\cdot \rho\cdot t^{2}}
\]
\end{proposition}
\begin{proof}
  First we observe that the fiber of the reduction $Y'_{\bar u,0}\to
  X'_{\bar u,0}$ over a point $x\in X'_{\bar u,0}$ is a homogeneous
  probability space of cardinality equal to $N(x,\bar u)$, therefore
  its entropy is $\ln N(x,\bar u)$.
  \begin{align*}
    \p&\set{\bar u\in U^{N}
         \st
         [Y'_{\bar u,0}\rel X'_{\bar u,0}]
         >
         \ln(N\cdot \rho)+t
        }
    \\
    \p&\set{\bar u\in U^{N}
         \st
         \int_{X'_{\bar u,0}}[Y'_{\bar u,0}\rel x]\d p_{X'_{\bar u,0}}(x)
         >
         \ln(N\cdot \rho)+t
        }
    \\
    &=
    \p\set{\bar u\in U^{N}
         \st
         \sum_{x\in X_{0}}
         \frac{\nu(x,\bar u)}{\rho\cdot|X_{0}|}
         \ln\big(N\cdot\nu(x,\bar u)\big)
         >
         \ln(N\cdot \rho)+t
    }
    \\
    &\leq
    \p\set{\bar u\in U^{N}
         \st
         \sum_{x\in X_{0}}
         \frac{\nu(x,\bar u)}{\rho\cdot|X_{0}|}
         \ln\big(\frac{\nu(x,\bar u)}{\rho}\big)
         >
         t
    }
    \\
    &\leq
    |X_{0}|\cdot
    \p\set{\bar u\in U^{N}
         \st
         \frac{\nu(x_{0},\bar u)}{\rho}
         \ln\big(\frac{\nu(x_{0},\bar u)}{\rho}\big)
         >
         t
    }
    \\
    &\leq
    |X_{0}|\cdot
    \ebf^{-\frac{1}{12}N\cdot \rho\cdot t^{2}}
  \end{align*}
The last inequality above follows from Lemma~\ref{p:binomial-totalvar}(ii).
\end{proof}

\subsection{Proof of Proposition~\ref{p:contraction-approximation}}
Let $\Xcal'_{\bar u}\ot \Ycal'_{\bar u}\to V_{N}$ be the fan
constructed in Section~\ref{s:construction}. The construction is
parameterized by number $N$ and atom $\bar u\in U^{N}$. Below we will choose
a particular value for $N$ and apply estimates in
Propositions~\ref{p:ikd-estimate} and~\ref{p:height-estimate} with
particular choice of parameter $t$ to show that there is $\bar u\in
U^{N}$, so that the fan satisfies conclusions of
Proposition~\ref{p:contraction-approximation}.

Let 
\begin{align*}
  N
  &:=
  \ln^{3}|X_{0}|\cdot \rho^{-1}
  =\ln^{3}|X_{0}|\cdot \ebf^{[X_{0}:U]}\\
  t&:=\frac{10}{\ln|X_{0}|}
\end{align*}
With this choices of $N$ and $t$ Proposition~\ref{p:ikd-estimate}
implies
\[
  \p\set{\bar u\in U^{N} \st \ikd(\Xcal'_{\bar
    u},\Xcal)>20\size{\Gbf}}\leq \frac{1}{4}
\]
while Proposition~\ref{p:height-estimate} gives
\[
  \p\set{\bar u\in U^{N} \st 
    [Y'_{\bar u,0}\rel X'_{\bar u,0}]>4\ln\ln |X_{0}|}\leq \frac{1}{4}
\]
Therefore there is a choice of $\bar u$ such that the fan 
\[
  \big(
    \Xcal'\ot \Ycal'\to V
  \big)
  :=
  \big(
    \Xcal'_{\bar u,0} \ot \Ycal'_{\bar u,0} \to V_{N}
  \big)
\]
satisfies conditions in~(\ref{eq:conditions}). 
As we have explained in the beginning of Section~\ref{s:construction},
by collapsing isomorphism arrows we obtain $\Gbf$-diagram $\Zcal'$ satisfying
conclusions of Proposition~\ref{p:contraction-approximation}.
\eproof

\subsection{Proof of Lemma~\ref{p:binomial-totalvar}}
The Chernoff bound for the tail of a binomially distributed random
  variable $X\sim\operatorname{Bin}(N,\rho)$ asserts that for any $0 \leq
  \delta \leq 1$ holds
  \begin{align*}
    \p\set{X<(1-\delta) N\cdot \rho}
    &\leq
    \ebf^{-\frac12\delta^2 N\cdot \rho }
    \\
    \p\set{X>(1+\delta) N\cdot \rho}
    &\leq
    \ebf^{-\frac13\delta^2 N\cdot \rho}
  \end{align*}
Applying the bound for upper and lower tail for the binomially
distributed random variable $N\cdot\nu$ we obtain the inequality in (i).

The second assertion follows from the following estimate
\begin{align*}
  \p\set{\frac{\nu}{\rho}\ln\frac{\nu}{\rho}>t}
  &\leq
  \p\set{\frac{\nu}{\rho}\big(\frac{\nu}{\rho}-1\big)>t}
  \\
  &=
  \p\set{\nu > \rho\cdot
  \left(
    \frac{\sqrt{1+4t}-1}{2}+1
  \right)
  }
\end{align*}
For $0\leq t\leq 2$ we have $\sqrt{1+4t}-1\geq t$, therefore
\begin{align*}
  \p\set{\frac{\nu}{\rho}\ln\frac{\nu}{\rho}>t}
  &\leq
  \p\set{\nu > \rho\cdot
  \left(
    \frac{t}{2}+1
  \right)
  }
\end{align*}
By the Chernoff bound we have
\begin{align*}
  \p\set{\frac{\nu}{\rho}\ln\frac{\nu}{\rho}>t}
  &\leq
  \ebf^{-\frac1{12}{N\cdot \rho\cdot t^{2}}}
\end{align*}
\eproof